\documentclass[11pt]{article}
\usepackage{amsmath,amsthm,amssymb,amscd,fancybox,epsfig,float}
\usepackage{graphicx}
\usepackage{epstopdf}
\usepackage{multirow}
\usepackage{hhline}
\usepackage{cleveref}
\usepackage{xcolor}
\usepackage{tikz}
\usetikzlibrary{arrows,automata}
\usepackage[nocompress]{cite}
\usepackage{url}
\usepackage{comment}

\usepackage{pgfplots}
\pgfplotsset{compat=1.17}   % as suggested by warning, why. Doesn't fix prob!

\usepackage{adjustbox}
\usetikzlibrary{arrows.meta}

\usepackage{subcaption}
\usepackage{authblk}
\usepackage{blindtext}
\newcommand{\hk}{\hat{k}}         % note \kh seems to not raise error, produce nothing :(

\newcommand{\inner}[2]{\ensuremath{\langle #1,#2 \rangle}}

\newcommand{\R}{\mathbb{R}}
\newcommand{\C}{\mathbb{C}}
\newcommand{\N}{\mathbb{N}}
\newcommand{\Z}{\mathbb{Z}}

\usepackage{bm}

\newcommand{\tran}{^{\mathsf{T}}}
\usetikzlibrary{positioning}

\newcommand{\bi}{\begin{itemize}}
\newcommand{\ei}{\end{itemize}}
\newcommand{\ben}{\begin{enumerate}}
\newcommand{\een}{\end{enumerate}}
\newcommand{\be}{\begin{equation}}
\newcommand{\ee}{\end{equation}}
\newcommand{\bea}{\begin{eqnarray}} 
\newcommand{\eea}{\end{eqnarray}}
\newcommand{\ba}{\begin{align}} 
\newcommand{\ea}{\end{align}}
\newcommand{\bse}{\begin{subequations}} 
\newcommand{\ese}{\end{subequations}}
\newcommand{\bc}{\begin{center}}
\newcommand{\ec}{\end{center}}
\newcommand{\bfi}{\begin{figure}}
\newcommand{\efi}{\end{figure}}
\newcommand{\ca}[2]{\caption{#1 \label{#2}}}
\newcommand{\ig}[2]{\includegraphics[#1]{#2}}
\newcommand{\bmp}[1]{\begin{minipage}{#1}}
\newcommand{\emp}{\end{minipage}}
\newcommand{\tbox}[1]{{\mbox{\tiny \rm #1}}}
\newcommand{\mbf}[1]{{\mathbf #1}}
\newcommand{\half}{\mbox{\small $\frac{1}{2}$}}
\newcommand{\bigO}{{\mathcal O}}

\newcommand{\eps}{\varepsilon}     % tolerance
\newcommand{\bj}{j}    %\mbf{j}}       % multiindex j.  keep non bold.
\newcommand{\xz}{x}                 % a new target point, aka \tilde{x} or x^*
\newcommand{\kz}{{\mbf{k}_\xz}}       % col vec of k from all x_j to target
\newcommand{\tkz}{{\tilde{\mbf{k}}_\xz}}       % col vec of k from all x_j to target, approx
\newcommand{\bal}{\bm{\alpha}}       % FS vec for mean
\newcommand{\tbal}{\tilde{\bm{\alpha}}}       % FS vec for mean, approx
\newcommand{\bbe}{\bm{\beta}}           % WS vec for mean
\newcommand{\bgz}{\bm{\gamma}_\xz}     % "
\newcommand{\bmu}{\bm{\mu}}          % vec of means at obs pts
\newcommand{\tbmu}{\tilde{\bm{\mu}}}          % vec of means at obs pts, approx
\newcommand{\X}{\Phi}                % design matrix aka X
\newcommand{\Xt}{\Phi^*}            % its conj.
\newtheorem{theorem}{Theorem}
\newtheorem{lem}[theorem]{Lemma}

\newtheorem{pro}[theorem]{Proposition}
\newtheorem{cor}[theorem]{Corollary}
\newtheorem{rmk}[theorem]{Remark}
\newtheorem{conj}[theorem]{Conjecture}
\newtheorem*{justification}{Justification of Conjecture \ref{c:Kheur}}

\title{Uniform approximation of common Gaussian process kernels
  using equispaced Fourier grids}

\author[1]{Alex Barnett\thanks{abarnett@flatironinstitute.org}}
\author[2]{Philip Greengard}
\author[1]{Manas Rachh}

\affil[1]{Center for Computational Mathematics, Flatiron Institute, New York, NY, 10010}
\affil[2]{Department of Statistics, Columbia University, New York, NY, 10027}

\date{\today}
\usepackage{cleveref}
\begin{document}
\maketitle

%\tableofcontents % for editing

\begin{abstract}
The high efficiency of a recently proposed method for computing with Gaussian  processes relies on expanding a (translationally invariant) covariance kernel into complex exponentials, with frequencies lying on a Cartesian equispaced grid.  Here we provide rigorous error bounds for this approximation for two popular kernels---Mat\'ern and squared exponential---in terms of the grid spacing and size. The kernel error bounds are uniform over a hypercube centered at the origin. Our tools include a split into aliasing and truncation errors, and bounds on sums of Gaussians or modified Bessel functions over various lattices. For the Mat\'ern case, motivated by numerical study, we conjecture a stronger Frobenius-norm bound on the covariance matrix error for randomly-distributed data points. Lastly, we prove bounds on, and study numerically, the ill-conditioning of the linear systems arising in such regression problems.
\end{abstract}

% IIIIIIIIIIIIIIIIIIIIIIIIIIIIIIIIIIIIIIIIIIIIIIIIIIIIIIIIIIIIIIIIIIIIIIIIIIIIIIIIII
\section{Introduction}
\label{s:intro}

Over the last couple of decades, Gaussian processes (GPs) have seen widespread use in statistics and data  science across a range of natural 
and social sciences \cite{rasmus1, bartok1, dfm1, cressie2018, bda, heaton2019}. 
In the canonical Gaussian process regression task, the goal is to recover 
an unknown real-valued function $f: D \subseteq \R^d \to \R$ using noisy observations of that 
function. Specifically, given data locations $x_1,\dots,x_N \in \R^d$, and 
corresponding observations $y_1,\dots, y_N \in \R$, the usual Gaussian process
regression model is 
\begin{align}
y_n & \;\sim\; f(x_n) + \epsilon_n, \qquad n=1,\ldots,N, \\
f(x) & \;\sim\; \mathcal{GP}(m(x), k(x, x')),
\end{align}
where $\mathcal{GP}$ denotes a Gaussian process distribution, 
$\epsilon_n \sim \mathcal{N}(0, \sigma^2)$ is independent and identically
distributed (iid) noise of known variance $\sigma^2>0$,
$m: \R^d \rightarrow \R$ is a given
prior mean function, and $k: \R^d \times \R^d \to \R$ is a given
positive definite covariance kernel \cite{rasmus1}.
In practice, $k$ is often also translation-invariant, that is,
$k(x, x') = k(x - x')$.

In general, the mean function $m$ can be set to zero by subtraction, which 
from now we will assume has been done.
Then, the marginal posterior of $f$ 
at any point $x \in D$ is Gaussian with mean $\mu(x)$ and variance $s(x)$ given by,
\bea
\mu(\xz) &=& \sum_{n=1}^{N} \alpha_n k(\xz,x_n)
%=: \bal\tran \kz
,
\label{mu}                 % was {223b} why
\\
s(\xz) &=&
k(\xz,\xz) - \sum_{n=1}^N  \gamma_{\xz,n} k(\xz, x_n)
%=:k(\xz,\xz) - \bgz\tran \kz
,
\label{s}
\eea
where $\bal := \{\alpha_n\}_{n=1}^N$
and $\bgz := \{\gamma_{\xz,n}\}_{n=1}^N$ are the vectors in $\R^N$
that uniquely solve the $N\times N$ symmetric linear systems
\bea
(K + \sigma^2 I) \bal &=& \mbf{y}
,
\label{fs_sys} \\
(K + \sigma^2I) \bgz &=& \kz
,
\label{bgz}
\eea
respectively,
where $K$ denotes the $N \times N$ positive semidefinite matrix with 
$K_{i, j} = k(x_i, x_j)$, and $\kz := \{ k(x,x_{n}) \}_{n=1}^{N}$.
%is a column vector.
These are known as ``function space'' linear systems
\cite{rasmus1}.

While GP regression has achieved widespread popularity, an inherent practical
limitation of the procedure is its computational cost.
A dense direct solution of the above linear systems requires
$\bigO(N^3)$ operations, and in the case of variance $s(x)$
a new right-hand side and solve is needed for each $x$.
Since in many modern data sets $N$ can be in the millions or more, a large 
literature has emerged on faster approximate methods for solving these
linear systems, and related tasks such as computing the determinant~\cite{RRGP05, nystrom, rasmus1, gardner1, wilson2, oneil1,minden1,stein1,chen1}.
An in-depth review of the computational environment for GP regression 
is outside the scope of this paper, though a summary can be found in, for example, \cite{heaton2019,liu20bigreview,efgp_comput}.

In this work, we analyze the equispaced Fourier Gaussian process (EFGP) regression approach recently proposed by the authors~\cite{efgp_comput}. Briefly, in EFGP, the covariance kernel is factorized as $k(x-x') \approx \tilde k(x-x') := 
\sum_{j=1}^M \phi_j (x)\phi_j(x')$ where the plane wave bases $\{\phi_{j}\}$ arise from an equispaced quadrature discretization of the
%truncated
inverse Fourier transform of the covariance kernel, using
$M = \bigO(m^d)$ nodes, where $m$ sets the grid size in each dimension.
This leads to a rank-$M$ approximation of the covariance matrix $K \approx \tilde{K} = \Phi \Phi^{*}$.
The method then proceeds to solve the equivalent ``weight space'' dual
system, with $M \times M$ system matrix $\Phi^{*} \Phi + \sigma^2I$,
using conjugate gradients (CG)
\cite{dahlquist1}. The method derives its computational efficiency from the ability to rapidly apply the Toeplitz matrix $\Phi^{*} \Phi$ using padded $d$-dimensional fast Fourier transforms (FFTs) with cost $O(M \log{M})$.
A precomputation which exploits nonuniform FFTs of cost $O(N + M\log{M})$ is needed;
however, the
cost per iteration is {\em independent} of the number of data points $N$.
The result is that for low-dimensional problems (say, $d\le 3$),
$N$ as high as $10^9$ can be regressed in minutes on a desktop; this is much
faster that competing methods in many settings \cite{efgp_comput}.

Error bounds for such a Fourier kernel approximation are
crucial in practice
in order to choose the numerical grid spacing $h$ and grid size $m$.
Then the error in the computed posterior mean when using $\tilde{k}$ as the covariance kernel can be bounded in terms of the Frobenius norm of $\tilde{K}-K$, which in turn can be bounded by the uniform kernel approximation error \cite[Thm.~4.4]{efgp_comput}.
This is a deterministic analysis of what is sometimes termed
``computational uncertainty'' \cite{wenger2022}.
This motivates us to derive error estimates for $\tilde k - k$, for the commonly-used
squared exponential (SE) and Mat\'ern kernels, with explicit dimension- and kernel-dependent constants.
Our bounds are uniform over a kernel argument lying in $[-1,1]^d$, as appropriate for
evaluating the kernel $\tilde k(x-x')$ for all $x,x'$ in the hypercube $D = [0,1]^d$.
We provide convenient explicit bounds on $h$ and $m$ that guarantee a
user-specified uniform error $\eps$ (Corollaries~\ref{c:SEparams} and \ref{c:matpars}).
Our Mat\'ern bounds are reminiscent of an analysis of %periodization for
Gaussian random field sampling by Bachmayr et al.~\cite{bachmayr20},
but we include kernel approximation error and our constants are explicit. % also we're simpler.
Since an equispaced tensor-product grid is perhaps the {\em simplest (deterministic) way to discretize a kernel in Fourier space}, we expect the bounds to
have wider applications to kernel approximations and Gaussian random fields.

Yet, we find that such bounds are in practice pessimistic for Mat\'ern kernels of low smoothness $\nu$,
due to the slow algebraic Fourier decay of the kernel.
To address this discrepancy we conjecture
a stronger bound on $\| \tilde K - K \|_{F}$ for
data points drawn iid randomly from some absolutely continuous measure.
%and show that it relates to a weighted $L^2$ error of the covariance kernel, in contrast to the rigorous $L^{\infty}$ result.
We support this with a brief derivation and a numerical
study. This provides a heuristic for choosing more efficient EFGP numerical parameters.
%that often outperform the parameters implied by the rigorous bounds. 

Finally, motivated by experiments \cite[Sec.~5]{efgp_comput} exhibiting
very large CG iteration counts, we include a preliminary
analysis and study of the condition numbers of the ``exact'' (true kernel $k$)
linear system, and the
approximate (kernel $\tilde k$) function and weight space systems.
While the issue of ill-conditioning of the function space system
is well known \cite{steinprecond,rudi17}
(and studied in the operator case \cite{wathenzhu15}
as well as in the $\sigma=0$ setting
of radial basis approximation \cite[Ch.~12]{wendland05book}),
the weight space system condition number is less well studied.
It turns out that both function- and weight-space linear systems
are nearly as ill-conditioned as their upper bounds allow (about $N/\sigma^2$),
even though the
the GP regression {\em problem}
%(at least when regressing to the data points themselves)
itself is very well-conditioned (Proposition~\ref{p:mucond}).
It is thus a curious situation from the perspective of numerical analysis
to have a well-conditioned problem require an ill-conditioned algorithm for its
solution
(compare, e.g., the unstable algorithm discussed in \cite[Ch.~15]{trefethen2}).
%It turns out that this is controlled by
%the instrinsic ill-conditioning of the exact function space system matrix
%in the large-$N$ and small-noise limit.
%Since it also affects GP regression error,
This twist %story
complements the main kernel error bounds of the paper.

%\subsection{Outline of this paper}
The remainder of this paper is structured as follows.
In~\Cref{s:prelim}, we review bounds on errors in posterior means using approximate GP regression, and provide a summary of the EFGP algorithm introduced
in \cite{efgp_comput}.
The main results are the approximation errors for the
SE and Mat\'ern kernels derived in~\Cref{s:discr}. In~\Cref{s:frob}, we 
conjecture a bound for the norm of $\tilde K-K$ in terms of a weighted $L^2$ approximation
of the covariance kernel, and give a heuristic derivation
along with numerical evidence.
We discuss bounds on various condition numbers, as well as a numerical study,
in~\Cref{s:cond}. We summarize and list some open questions in~\Cref{s:conc}.

\section{Preliminaries \label{s:prelim}}
In this section,
we motivate the study of the kernel approximation error
by reviewing how it controls the error in the posterior mean
(relative to exact GP regression with the true kernel).
We also summarize the EFGP numerical method.
Both are presented in more depth in \cite{efgp_comput}.

\subsection{Error estimates for the posterior mean}
Suppose that $\tilde{k}$ is an approximation to $k$ with a uniform error $\varepsilon$, i.e.
\begin{equation*}
\sup_{x,x' \in D} |k(x,x') - \tilde{k}(x,x')| \le \varepsilon,
\end{equation*} 
then, since all data points lie in $D$,
the error in the corresponding covariance matrix is easily bounded by
\begin{equation}
\| K- \tilde{K} \| \leq  \| K -\tilde{K} \|_{F} \leq N \varepsilon,
\label{Kerr}
\end{equation}
where $\| \cdot \|$ denotes the spectral norm of the matrix, and $\| \cdot \|_{F}$ denotes the Frobenius norm.

Furthermore, let $\bal$, and $\tbal$ be the solutions to 
\begin{equation}
(K + \sigma^2 I) \bal = \mbf{y}, \qquad (\tilde{K} + \sigma^2 I ) \tbal = \mbf{y} ,
\end{equation}
and let $\bmu = K \bal$, and $\tbmu = \tilde{K} \tbal$ be the corresponding posterior mean vectors at the observation points.
Then the error in this posterior mean vector satisfies
\begin{equation}
  \frac{\| \bmu - \tbmu \|}{\|\mbf{y} \|} \leq \frac{\| K -\tilde{K} \|}{\sigma^2} \leq \frac{N \varepsilon}{\sigma^2}.
  \label{muerr}
\end{equation}

Finally, let $\mu(\xz) = \mbf{k}_\xz\tran \bal$,
where $\kz := [k(\xz, x_1), \dots, k(\xz, x_N)]\tran$,
be the true posterior mean at a new test target $\xz\in\R^d$,
and let $\tilde\mu(\xz) = \tkz\tran \tbal$ be its approximation.
Then its error (scaled by the root mean square data magnitude $\|\mbf{y}\|/\sqrt{N}$) obeys
\bea\label{muerrnew}
  \frac{|\tilde\mu(\xz)-\mu(\xz)|}{\|\mbf{y}\|/\sqrt{N}}
&\le&
\biggl(\frac{N^2}{\sigma^4} + \frac{N}{\sigma^2} \biggr) \eps .
\eea
These results
(simplifications of \cite[Thm.~4.4]{efgp_comput})
show that it suffices to bound $\eps$, the uniform approximation error of the covariance kernel, in order to bound the error in computed posterior means.

% OOOOOOOOOOOOOOOOOOOOOOOOOOOOOOOOOOOOOOOOOOOO
\subsection{Summary of the EFGP numerical scheme for GP regression}
\label{s:efgp}

Suppose that $k: \mathbb{R}^{d} \to \mathbb{R}$ describes a
translation-invariant and integrable covariance kernel $k(x-x')$.
In EFGP, this kernel is approximated by discretizing the Fourier transform of the
covariance kernel using an equispaced quadrature rule. 
Specifically, using the Fourier transform convention of \cite{rasmus1},
we have
\begin{align}
  & \hat{k}(\xi) = \int_{\R^{d}} k(x) e^{-2\pi i \inner{\xi}{x}}\,  dx,
\qquad   \xi \in \R^{d},
  \\
  & k(x) = \int_{\R^{d}} \hat{k}(\xi) e^{2\pi i \inner{\xi}{x}} \, d\xi,
\qquad x \in \R^{d}.
  \label{inv_ft}
\end{align}
Discretizing \eqref{inv_ft} with an equispaced trapezoid tensor-product
quadrature rule we obtain
\begin{align}\label{ker_eff}
  k(x-x') \;\approx\;
  \tilde{k}(x - x') = \sum_{\bj \in J_{m}} h^d \hk(h\bj) e^{2\pi i h \inner{\bj}{x - x'}} ,
\end{align}
where the multiindex $\bj:=(j^{(1)},j^{(2)},\dots,j^{(d)})$
has elements $j^{(l)} \in \{-m,-m+1,\dots,m\}$ and thus
ranges over the tensor product set
$$
J_{m} := \{-m,-m+1,\dots,m\}^d
% simpler than
%\{ (j^{(1)},\ldots j^{(d)}) : j^{(l)} \in \left\{ -m, \ldots m \right\} \,, l=1,2,\dots, d \}
$$
containing $M=(2m+1)^d$ elements.
%This quadrature is simply the {\em trapezoidal rule}
%\footnote{Note that the usual 1D trapezoid rule definition has factors of $1/2$ applied to the
%first and last points \cite[p.~209]{na}; this difference is inconsequential
%in our application since errors will instead be dominated by truncation to the box.}.
Splitting the exponential in \eqref{ker_eff}
we get the rank-$M$ symmetric factorization for the approximate kernel
\begin{equation}
\tilde{k}(x,x') = \sum_{j \in J_{m}} \phi_j(x)\overline{\phi_j(x')} \, ,
\end{equation}
with basis functions
$\phi_{j}(x) := \sqrt{h^d \hat{k}(hj)}e^{2\pi i h \inner{\bj}{x}}$.
Inserting the data points $\{x_n\}_{n=1}^N$ shows that $\tilde K = \X \Xt$,
where the design matrix $\X$ has elements $\X_{nj} = \phi_j(x_n)$.
Then in EFGP one solves the $M$-by-$M$ weight-space system
\be
(\Xt\X + \sigma^2I)\bbe = \Xt\mbf{y}
\label{WS}
\ee
iteratively using CG.
Its right-hand side vector can be filled by observing that
$\sum_{n=1}^N e^{2\pi i h\inner{\bj}{x_n}} y_n$ takes the form of a
type 1 $d$-dimensional nonuniform discrete Fourier transform, which may
be approximated in $\bigO(N + M\log M)$ effort via standard nonuniform FFT
(NUFFT) algorithms \cite{dutt1}.  % barnett1 ?
Since $(\Xt\X)_{j,j'}$ depends only on $j-j'$,
then $\Xt\X$ is a Toeplitz matrix, and
its Toeplitz vector can be computed by another NUFFT.
With these two $N$-dependent precomputations done,
the application of $\Xt\X$ in each CG iteration is a discrete nonperiodic
convolution, so may be performed by a standard padded $d$-dimensional FFT.
Finally, once an approximate solution vector $\bbe := \{\beta_j\}_{j\in J_m}$ is found,
the posterior mean $\mu(x) = \sum_{j\in J_m} \beta_j \phi_j(x)$ may be rapidly
evaluated at a large number of targets $x$, now via a type 2 NUFFT.
This weight-space formula for $\mu$ is equivalent to a function-space
solution of \eqref{fs_sys} with $K$ replaced by its approximation $\tilde K$ (see, e.g., \cite[Lem.~2.1]{efgp_comput}).
The posterior variance $s(x)$ may be found similarly by iterative solution of \eqref{bgz}, then evaluating \eqref{s}.

Note that the equispaced Fourier grid---being the root cause of the Toeplitz structure---is
crucial for the efficiency of EFGP.
This motivates the study of the kernel approximation properties of such a
Fourier grid, the subject of the next section.

%UUUUUUUUUUUUUUUUUUUUUUUUUUUUUUUUUUUUUUUUUUUUUUU
\section{Uniform bounds on the kernel discretization error}
\label{s:discr}
We now turn to the main results:
we derive explicit error estimates for the equispaced Fourier kernel approximation
in \eqref{ker_eff} in all dimensions $d$
for two families of commonly-used kernels: Mat\'ern and 
squared-exponential.
We assume that the source $x$ and target $x'$ are contained in the set $D=[0,1]^{d}$,
as appropriate when all data and evaluation points lie in this set.
Note that the coordinates may always be shifted and scaled to make this so.

We start by restating an exact formula for the error,
by exploiting the equispaced nature of the Fourier grid
(see \cite{efgp_comput}; for convenience we include the simple proof.)
\begin{pro}[Pointwise kernel approximation]  % pppppppppppppppppppppppppppp
  \label{p:alias_trunc}
 Suppose that the translationally invariant covariance kernel $k: \R^d \to \R$
  and its Fourier transform $\hat{k}$
  decay uniformly as
  $|k(x)|\le C(1+\|x\|)^{-d-\delta}$ and
  $|\hat{k}(\xi)|\le C(1+\|\xi\|)^{-d-\delta}$ for some $C$, $\delta>0$.
  Let $h>0$, $m\in\N$, then define $\tilde k$ by \eqref{ker_eff}.
  Then for any $x \in \R^d$ we have
  \begin{align}
    \label{alias_trunc}
  \tilde{k}(x) - k(x)
 \;\;=\;\;
  -
  \underbrace{\sum_{n \in \Z^d, \, n \neq \mbf{0}} k\left(x + \frac{n}{h} \right)}
  _\text{aliasing error}
  \;\;+\;\;
  \underbrace{h^{d} \sum_{j \in \mathbb{Z}^{d}, \, j \not \in J_{m}} \hat{k}(j h)e^{2\pi i h \inner{j}{x}}}_\text{truncation error}
  .
\end{align}
\end{pro}
\begin{proof}
  Writing $x$ in place of $x-x'$ in \eqref{ker_eff} gives
  $\tilde{k}(x) = h^{d} \sum_{j \in J_{m}} \hat{k}(j h)e^{2\pi i h  \inner{j}{x} }$.
  Shifting and scaling the Poisson summation formula
  \cite[Ch.~VII, Cor.~2.6]{steinweissbook} to give the form
\begin{align}\label{psf}
h^{d} \sum_{j \in \mathbb{Z}^{d}} \hat{k}(jh)e^{2\pi h i \inner{j}{x}} = 
\sum_{n \in \Z^{d}} k\biggl(x + \frac{n}{h} \biggr),
\end{align}
then splitting off the $n=\mbf{0}$ term on the right, and $J_m$ terms on the left,
completes the proof.
\end{proof}
Thus the error has two contributions, as illustrated in Fig.~\ref{f:discr}.
The {\it aliasing error} takes the form of a lattice sum of periodic images
(translates) of the kernel $k$, excluding the central element; see
Fig.~\ref{f:discr}(a).
Their separation is $h^{-1}$, and once this is a few times $\ell$ larger than
$1$, the exponential decay of the kernel ensures that this
term is uniformly small over $x\in[-1,1]^d$, the set $D-D$ of values
taken by $x-x'$. This set is shown by a black box in the plot.
The {\it truncation error}, the second term on the right-hand side of \eqref{alias_trunc},
arises due to limiting the Fourier integral to the finite box $[-mh,mh]^d$.
It is a tail sum of $\hk$ over the infinite lattice minus the finite box
that is summed computationally; see Fig.~\ref{f:discr}(b).
Once $h$ is determined by the aliasing error,
the truncation error may be made small by 
choosing $m$ such that the tail integral of $\hk(\xi)$ 
is small over the exterior of $[-mh, mh]$.

In practice, $h$ is set to the largest permissible value which achieves a certain aliasing error, then
$m$ is chosen according to the decay of $\hk$ to achieve a truncation error of the
same order.

We now apply the above to uniformly bound the error for approximating
two common kernels defined as follows \cite{rasmus1}.
Note that the value at the origin for both kernels is $k(\mbf{0})=1$,
appropriate for when the data has been scaled for unit prior covariance:
\bi
\item The squared exponential kernel with length scale $\ell$
  (using $|\cdot|$ for Euclidean norm),
\begin{equation}
  G_\ell(x) :=  \exp{\left(-\frac{|x|^2}{2\ell^2} \right)}\, .
  \label{Gker}
\end{equation}
\item
The Mat\'ern kernel with smoothness parameter $\nu\ge 1/2$ and length scale $\ell$,
\begin{equation}
  C_{\nu,\ell}(x) := \frac{2^{1-\nu}}{\Gamma(\nu)} \left(\sqrt{2\nu} \frac{|x|}{\ell} \right)^{\nu} K_{\nu} \biggl(  \sqrt{2\nu}\frac{|x|}{\ell}\biggl) \, ,
  \label{Cker}
\end{equation}
where $K_{\nu}$ is the modified Bessel function of the second kind.
\ei

\bfi   % ffffffffffffffffffffffffffffffffffffffffffffffffffffffffffffffffffff
\centering
\ig{width=\textwidth}{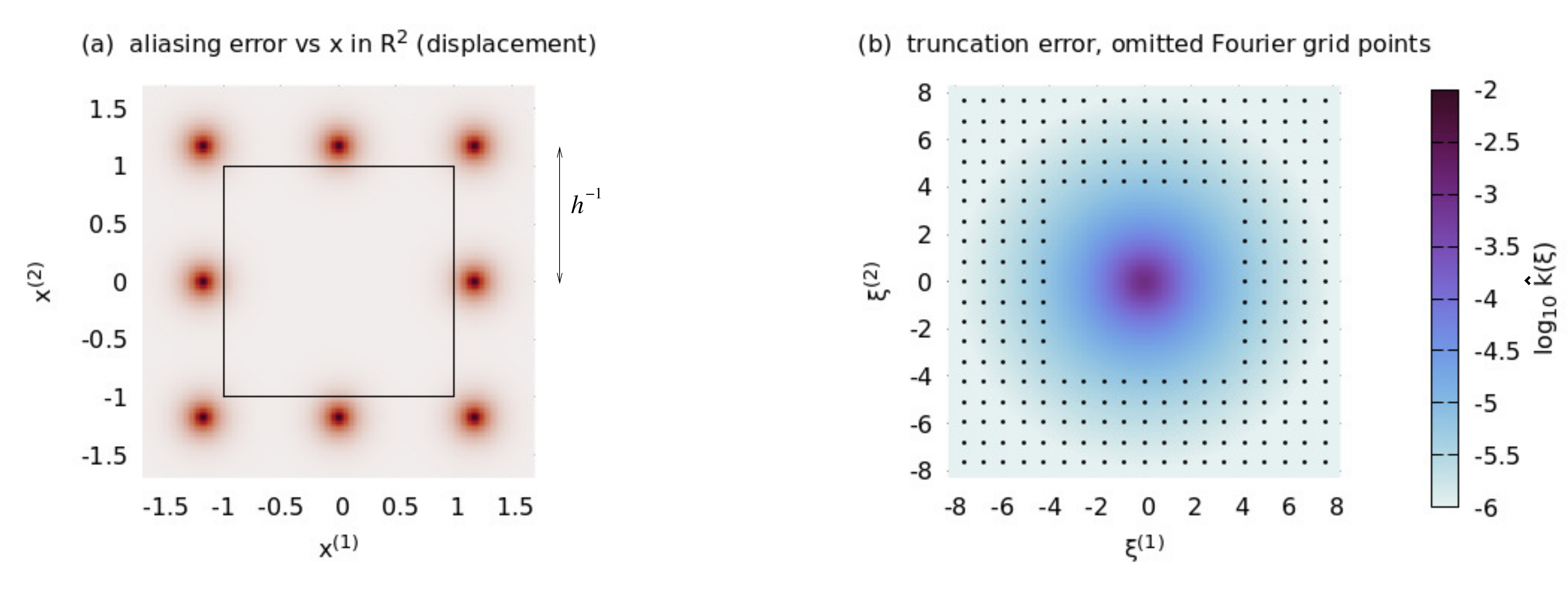}
\vspace{-5ex}
\ca{Illustration of two contributions to the discretization error in the pointwise approximation of the kernel $k(x)$, in $d=2$; see Section~\ref{s:discr}.
  Panel (a) shows the aliasing error term imaged as a function of the displacement
  argument $x$. The black square shows the domain $[-1,1]^2$.
  Panel (b) images $\hk(\xi)$ (on a logarithmic color scale) in the Fourier plane,
  and shows as dots the punctured infinite lattice of excluded Fourier frequencies
  $hj$, where $j\in\Z^2$, $j\notin J_m$, and $J_m$ is the $(2m+1)$-by-$(2m+1)$
  grid of quadrature nodes (not shown).
  The truncation error is bounded by the sum of $\hk$ at all dots.
%  (The numerically-summed $2m+1$ by $2m+1$ grid about the origin is not shown.)
  The parameters (chosen merely for visual clarity)
  are $h=0.85$, $m=4$, for a Mat\'ern kernel with $\nu=1/2$.
}{f:discr}
\efi

% SSSSSSSSSSSSSSSSSSSSSSSSSSSSSSSSSSSSSSSSSSSSSSSSSSSSSSSSSSSSSSSSSSSSSSSSSSSSSSSSSSS
\subsection{Squared-exponential kernel}

Recall that for data points lying in $D=[0,1]^d$, the kernel $k(x)$ must be
well approximated over $x\in[-1,1]^d$.
The theorem below gives uniform bounds for the
two contributions to the error.
The result shows superexponential convergence both in $h$ (once $h<1$),
and in $m$.
In practice, for the typical case of $\ell\ll1$,
machine accuracy ($\approx 10^{-16}$) is
reached once $h$ is less than 1 by a few times $\ell$,
and $m$ is a couple times $1/\ell$.

In the proof below, the following elementary bounds on Gaussian sums are useful.
For any $a>0$, we have by monotonicity,
\be
\sum_{j=1}^\infty e^{-(a j)^2/2}
\le \int_0^\infty e^{-a^2 t^2/2} dt = \frac{\sqrt{\pi}}{2a}
\label{gau1}
\ee
and hence
\be
\sum_{j\in\Z} e^{-(a j)^2/2} = 1 + 2\sum_{j=1}^\infty e^{-(a j)^2/2} \le 1 + \frac{\sqrt{\pi}}{a}.
\label{gauZ}
\ee
We also need the Fourier transform of $G_\ell$ in \eqref{Gker},
using the convention \eqref{inv_ft},
\be
\hat{G}_\ell(\xi) = (\sqrt{2\pi}\ell)^d e^{-2|\pi\ell\xi|^2}
.
\label{Ghat}
\ee

\begin{theorem} % tttttttttttttttttttttttttttttttttttttttttttttttttttttttttttttttttt
  [Aliasing and truncation error for squared-exponential covariance kernel]
  \label{t:Gerr}
Suppose that $k(x) = G_\ell(x)$ as  defined by \eqref{Gker}, with length scale
$\ell\le 2/\sqrt{\pi} \approx 1.13$.
Let $h<1$ be the frequency grid spacing.
Then the aliasing error magnitude is bounded uniformly over $x \in [-1,1]^{d}$ by
\be
\Biggl|
\sum_{\substack{n \in \mathbb{Z}^{d} \\ n \neq \mbf{0}}} k\left(x + \frac{n}{h} \right)
\Biggr|
  \;\leq\;
  2d\,3^de^{-\frac{1}{2}\left( \frac{h^{-1}-1}{\ell} \right)^2}
    .
  \label{Galias}
  \ee
  In addition, letting $m\in\N$ control the grid size ($2m+1$ in each dimension),
  the truncation error magnitude is bounded uniformly over $x\in\R^d$ by
  \be
  \Biggl|
  h^d
\sum_{\substack{j \in \mathbb{Z}^{d} \\ j \notin J_m}}
 \hat{k}(j h)e^{2\pi i h \inner{j}{x}}
\Biggr|
  \;\leq\;
2d \,4^d e^{-2(\pi\ell h m)^2}
    .
  \label{Gtrunc}
  \ee
\end{theorem}

\begin{proof}[Proof of Theorem~\ref{t:Gerr}]
  We first bound the aliasing error,
  by
  exploiting the fact that it is uniformly bounded over $[-1,1]$ by
  its value at $(1,0,\dots,0)$.
  % uniformly over $x\in[-1,1]^d$.
  Noting that $G_\ell$ is positive and isotropic,
  the left side of \eqref{Galias}
  is bounded by $2d$ equal sums over overlapping half-space
  lattices (pointing in each of the
  positive and negative coordinate directions),
  \bea
  %  \eps_\text{trunc} :=
  &&\max_{x\in[-1,1]^d}
  \Biggl|
  \sum_{\substack{n \in \Z^{d} \\ n \neq \mbf{0}}} G_\ell\left(x - \frac{n}{h}\right)
  \Biggr|
  \;\le\;
  2d \max_{x\in[-1,1]^d} \sum_{p=1}^\infty \sum_{q\in \Z^{d-1}} G_\ell\left(x - \frac{(p,q)}{h}\right)
  \nonumber
\\
&&\qquad=
2d \,
\biggl(\max_{s\in[-1,1]} \sum_{p=1}^\infty e^{-(s-p/h)^2/2\ell^2} \biggr)
\biggl(\max_{t\in[-1,1]} \sum_{q\in\Z} e^{-(t-q/h)^2/2\ell^2} \biggr)^{d-1}
\label{Gsep}
\eea
where in the second line we split off $s$ as the first coordinate of $x$, and
used separability of the Gaussian.
The sum over $q$ is bounded by its value for $t=0$,
because, by the Poisson summation formula \eqref{psf}
and \eqref{Ghat},
this sum is equal for any $t\in\R$
to $h \sum_{j\in\Z} e^{2\pi i t h j}\sqrt{2\pi}\ell e^{-2(\pi\ell h j)^2}$.
%which is bounded by the case $t=0$.
Then setting $t=0$, this sum is bounded by using $a=1/h\ell$ in \eqref{gauZ}
to give
\be
\max_{t\in[-1,1]} \sum_{q\in\Z} e^{-(t-q/h)^2/2\ell^2} \le
\sum_{q\in\Z} e^{-q^2/2h^2\ell^2} \le 1 + \sqrt{\pi}\ell h
.
\label{Zbnd}
\ee
However, the first sum over $p$ in \eqref{Gsep} is bounded
by its value at $s=1$, which can be seen because $h<1$ thus each term
is monotonically increasing in $s$.
Then by writing $(1-p/h)^2 = [(p-1)h^{-1} + (h^{-1}-1)]^2 =
(h^{-1}-1)^2  + (p-1)^2h^{-2} + 2(h^{-1}-1)h^{-1}(p-1)$ and
noting that the last term is nonnegative,
$$
\sum_{p=1}^\infty e^{-(1-p/h)^2/2\ell^2}
\;\le\;
e^{-\half\left(\frac{h^{-1}-1}{\ell}\right)^2}
  \sum_{p=1}^\infty e^{-(p-1)^2/2h^2\ell^2}
\;\le\;
e^{-\half\left(\frac{h^{-1}-1}{\ell}\right)^2}
\left(1 + \frac{\sqrt{\pi}\ell h}{2}\right),
$$
where \eqref{gau1} was used with $a = 1/h\ell$ in the last step.
Inserting this and \eqref{Zbnd} into \eqref{Gsep}
and using $1 + \sqrt{\pi}\ell h \le 3$, implied by the hypotheses
$h<1$ and $\ell \le 2/\sqrt{\pi}$, finishes the proof of \eqref{Galias}.

The proof of the truncation error bound is similar because $\hat{G}_\ell(\xi)$
in \eqref{Gker} is also Gaussian.
Because $\hk$ is always nonnegative, the
left side of \eqref{Gtrunc} is bounded by its value at $x=\mbf{0}$.
As with the aliasing error, we may now bound the sum over the punctured
lattice by that over $2d$ half-space lattices,
\bea
&&
h^d \sum_{\substack{j \in \mathbb{Z}^{d} \\ j \notin J_m}}
\hat{k}(j h)
\;=\;
(\sqrt{2\pi}\ell h)^d \sum_{\substack{j \in \mathbb{Z}^{d} \\ j \notin J_m}}
e^{-\half|2\pi\ell h j|^2}
\nonumber
\\
&&\le\;
2d (\sqrt{2\pi}\ell h)^d \biggl(\sum_{p>m} e^{-\half(2\pi\ell h p)^2} \biggr)
\biggl(\sum_{q\in\Z} e^{-\half(2\pi\ell h q)^2} \biggr)^{d-1}
\label{Ghatsep}
\eea
The $q$ sum is bounded by $1+1/(2\sqrt{\pi}\ell h)$,
by choosing $a=2\pi\ell h$ in \eqref{gauZ}.
The $p$ sum is bounded
by writing $p=m+j$, and dropping the nonnegative
last term in $p^2=(m+j)^2 =m^2+j^2 +2mj$,
then using \eqref{gau1}, so
$$
\sum_{p>m} e^{-\half(2\pi\ell h p)^2}
\;\le\;
e^{-2(\pi\ell h m)^2} \sum_{j=1}^\infty e^{-\half(2\pi\ell h j)^2}
\;\le\;
e^{-2(\pi\ell h m)^2} \left( 1 + \frac{1}{4\sqrt{\pi}\ell h} \right)
.
$$
Replacing $4$ by $2$ in the above, then
inserting these two bounds into \eqref{Ghatsep}
gives
$$
h^d \sum_{\substack{j \in \mathbb{Z}^{d} \\ j \notin J_m}}
\hat{k}(j h)
\;\le\;
2d\biggl(\sqrt{2\pi}\ell h + \frac{1}{\sqrt{2}} \biggr)^d
e^{-2(\pi\ell h m)^2}.
$$
The hypotheses $h<1$ and $\ell\le 2/\sqrt{\pi}$
guarantee that the factor taken to the $d$th power is no more than
$2\sqrt{2} + 1/\sqrt{2}<4$, proving \eqref{Gtrunc}.
\end{proof}

The above leads to the following simple rule to set $h$ and $m$ to guarantee a user-defined absolute kernel approximation error, in exact arithmetic.

\begin{cor}[Discretization parameters $(h,m)$ to guarantee uniform SE kernel
    accuracy $\eps$]
  \label{c:SEparams}
  Let $k=G_\ell$ be the SE kernel, and
  let $\ell\le 2/\sqrt{\pi}$ as above.
  Let $\eps>0$.
  Set $h \le \big(1 + \ell \sqrt{2 \log (4d\,3^d/\eps)}\big)^{-1}$
  then the aliasing error is no more than $\eps/2$.
  In addition, set $m \ge \sqrt{\half \log (4^{d+1}d/\eps)}/\pi\ell h$,
  then the truncation error is no more than $\eps/2$,
  so that $|\tilde k(x)-k(x)| \le \eps$ uniformly over $x \in [-1,1]^d$.
\end{cor}

% MMMMMMMMMMMMMMMMMMMMMMMMMMMMMMMMMMMMMMMMMMMMMMMMMMMMMMMMMMMMMMMMMMMMMMMMMMMMMMMMMMMM
\subsection{Mat\'ern kernel}

In this section we provide proofs for the aliasing error and truncation error estimates 
for the Mat\'ern kernel given by \eqref{Cker}. Its Fourier transform is
\be
\hat{C}_{\nu,\ell}(\xi) = \hat{c}_{d,\nu} \left(\frac{\ell}{\sqrt{2\nu}} \right)^d \left(2\nu + |2\pi\ell\xi|^2  \right)^{-\nu-d/2},
\label{Crecall}
\ee
where $|\cdot|$, as before, denotes Euclidean norm, and where
the prefactor $\hat{c}_{d,\nu}$ is
\begin{equation}
  \hat{c}_{d,\nu} = \frac{2^d \pi^{d/2} (2\nu)^{\nu} \Gamma(\nu+d/2)}{\Gamma(\nu)} .
  \label{hatc}
\end{equation}
In order to prove the estimate for the aliasing error, we state some decay properties of the modified Bessel function $K_\nu(z)$
(using \cite[10.29 and 10.37]{dlmf}
\cite[\S 8.486]{gradshteyn2014table}).
For $z>0$, and fixed $\nu$, $K_{\nu}(z)$ is monotonically decreasing and positive.
For fixed $z$, the modified Bessel functions are monotonically increasing in $\nu$,
i.e. $K_{\nu}(z) \leq K_{\mu}(z)$ for $\mu\geq \nu$. Moreover, 
\begin{equation}
\frac{d}{dz} (z^{\nu} K_{\nu}(z)) = -z^{\nu} K_{\nu-1}(z) = -z^{\nu} \left( K_{\nu+1}(z) - \frac{2\nu}{z}K_{\nu}(z) \right) .
\end{equation}
Note that the positivity of $K_{\nu-1}(z)$ implies that $z^{\nu} K_{\nu}(z)$ is also a monotonically decreasing function of 
$z$. 
The monotonicity properties and the positivity of $K_{\nu}$ also imply that
\begin{equation}
  \frac{1}{z^{\nu} K_{\nu}(z)}\frac{d}{dz} (z^{\nu} K_{\nu}(z)) \;\leq\; -\frac{1}{2} \, ,
  \qquad \forall z\geq 4\nu,
\end{equation}
related to a special case in \cite[Lem.~3]{bachmayr20}.
Integrating the equation in $z$, we get the exponential upper bound
\begin{equation}
\label{eq:besselprop}
f_\nu(z) := 
z^{\nu} K_{\nu}(z) \;\leq\; f_{\nu} (4\nu) e^{2\nu} e^{-z/2}\, ,\qquad z\ge 4\nu.
\end{equation}
Noting that the Mat\'ern kernel is proportional to $f_\nu(\sqrt{2\nu}|x|/\ell)$,
this places a useful exponential decay bound on the kernel beyond
a few $\ell$ away from its origin.
Note that our bound is on
$f_\nu(z)$ rather than $K_\nu(z)$ as in \cite[Lem.~2]{bachmayr20},
at the cost of a lower bound on $z$ and halving the exponential rate.

In order to prove the estimate for the truncation error in the following 
theorem, we first
need the following lemma bounding power-law half-space lattice sums.
\begin{lem}
  \label{l:powersum}
  Let $\nu>0$, let $d\ge 1$, and let $m\ge 1$. Then
\be
I(d,\nu,m) := \sum_{n>m} \sum_{q \in \Z^{d-1}} \left( n^2 + |q|^2 \right)^{-\nu -d/2} \leq \frac{\beta(d,\nu)}{m^{2\nu}}
% *** corrected sign of power!
\ee
where for fixed $\nu$ the
prefactor $\beta$ obeys the following recursion relation in dimension $d$,
  \be
\beta(d,\nu) = \left\{\begin{array}{ll}
  \frac{1}{2\nu},& d=1\\
  \left(4 + \frac{2}{2\nu + d-1} \right) \beta(d-1,\nu), & d>1.
  \end{array} \right.
  \label{pre}
  \ee
  In particular, for any $\nu\ge 1/2$ we have
  \be
  \beta(d,\nu) \le \frac{5^{d-1}}{2\nu},  \qquad d=1,2,\dots
  \label{beta5}
  \ee
\end{lem}

\begin{proof} 
  We observe for the case $d=1$ (upper case in \eqref{pre}),
  \be
  \sum_{n>m} n^{-2\nu-1} \le \int_m^\infty y^{-2\nu-1} dy = \frac{m^{-2\nu}}{2\nu}
  \label{1sum}
  \ee
  where monotonic decrease of the function was used to bound the sum by an integral.
  Now for $d>1$,
  $$
  \sum_{n>m} \sum_{q \in \Z^{d-1}} \left( n^2 + |q|^2 \right)^{-\nu -d/2} = \sum_{n>m} \sum_{w \in \Z^{d-2}} \sum_{q \in \Z} \left( n^2 + |w|^2 + q^2 \right)^{-\nu -d/2}
  ,
  $$
  where in the case $d=2$ we abuse notation slightly: in that case the sum over
  $w$ is absent.
  We split the innermost sum into the central part
  $q \leq \lceil \sqrt{n^2 + |w|^2} \rceil$, 
  where $\lceil x \rceil$ denotes the smallest integer not less than $x$,
  plus the two-sided tail $q > \lceil \sqrt{n^2 + |w|^2} \rceil$.
  The central part
  contains at most $2(\sqrt{n^2+|w|^2}+1) +1 < 4\sqrt{n^2+|w|^2}$
  terms, where this upper bound follows since $n\ge 2$,
  and each such term is bounded by the constant $(n^2+|w|^2)^{-\nu-d/2}$.
  The two-sided tail
  is bounded by
  $2\sum_{q > \lceil \sqrt{n^2 + |w|^2} \rceil} (q^2)^{-\nu-d/2} \le
  2\int_{\sqrt{n^2+|w|^2}}^\infty y^{-2\nu-d} dy = (2\nu+d-1)^{-1} (n^2+|w|^2)^{-\nu-(d-1)/2}$.
  Combining both of these estimates, we get
 \bea
   I(d,\nu,m ) &\leq&
   \left( 4 + \frac{2}{2\nu + d-1}\right) \sum_{n>m} \sum_{w \in Z^{d-2}} \left( n^2 + |w|^2 \right)^{-\nu -(d-1)/2}
   \nonumber \\
   &=& \left( 4 + \frac{2}{2\nu + d-1}\right) I(d-1,\nu,m)
   . \nonumber
   \eea
Recursing down in $d$, we get \eqref{pre}, from which  \eqref{beta5} follows immediately.
\end{proof}

We now present the main result: uniform bounds on the
two contributions to the error for the Mat\'ern kernel.
The following shows exponential convergence with respect to $h$ for the
aliasing error, but
only order-$2\nu$ algebraic convergence with respect to $m$ for the truncation error.
The latter is due to the algebraic tail of $\hat{C}_{\nu,\ell}(\xi)$.

\begin{theorem}  
[Aliasing and truncation error for the Mat\'ern covariance kernel]
\label{t:Cerr}
Suppose $k(x) = C_{\nu,\ell}(x)$ as in \eqref{Cker},
with smoothness $\nu\ge 1/2$ and
length scale $\ell\le (\log 2)^{-1}\sqrt{\nu/2d}$.
Let $h \le (1+\sqrt{8\nu}\ell)^{-1}$ be the frequency grid spacing.
Then the aliasing error magnitude is bounded uniformly over $x \in [-1,1]^{d}$ by
\be
\Biggl|
\sum_{\substack{n \in \mathbb{Z}^{d} \\ n \neq \mbf{0}}} k\left(x + \frac{n}{h} \right)
\Biggr|
\;\leq\;
4d\, 3^{d-1} \cdot \frac{2^{1-\nu}}{\Gamma(\nu)}(4\nu)^\nu e^{2\nu} K_\nu(4\nu)
\cdot
e^{-\sqrt{\frac{\nu}{2d}}\frac{h^{-1}-1}{\ell}}
.
\label{Calias}
\ee
In addition, letting $m\in\N$ control the grid size ($2m+1$ in each dimension),
the truncation error magnitude is bounded uniformly over $x\in\R^d$ by
\be
\Biggl|
h^d
\sum_{\substack{j \in \mathbb{Z}^{d} \\ j \notin J_m}}
\hat{k}(j h)e^{2\pi i h \inner{j}{x}}
\Biggr|
\;\leq\;
\frac{\nu^{\nu-1}d\, 5^{d-1}}{2^\nu \pi^{d/2 + 2\nu}}
\frac{\Gamma(\nu+1/2)}{\Gamma(\nu)}
\frac{1}{(h \ell m)^{2\nu}}.
\label{Ctrunc}
\ee
\end{theorem}

\begin{proof}[Proof of aliasing bound \eqref{Calias}]  % ppppppppppppppppppppppppppppp
As with the squared-exponential case, we note that the sum over $n \in Z^{d}\setminus \{\mbf{0}\}$ is bounded by $2d$ half-spaces of the form $p\geq 1$, $q \in \Z^{d-1}$, with $n=(p,q)$.
Owing to the radial symmetry of the kernel, all of those half spaces can be bounded using the same estimate.
Since $C_{\nu,\ell}(x)$ is positive we may remove absolute value signs.
Substituting \eqref{Cker}, and splitting $x=(s,t)$ where $s$ is the first
coordinate and $t\in\R^{d-1}$,
\bea
%  \eps_\text{trunc} :=
&&\max_{x\in[-1,1]^d}
\Biggl|
\sum_{\substack{n \in \Z^{d} \\ n \neq \mbf{0}}} C_{\nu,\ell}\left(x - \frac{n}{h}\right)
\Biggr|
\;\le\;
2d
\max_{x\in[-1,1]^d}
  \sum_{p=1}^\infty \sum_{q\in \Z^{d-1}} C_{\nu,\ell}\left(x - \frac{(p,q)}{h}\right)
\nonumber
\\
&&\quad =\; 2d \frac{2^{1-\nu}}{\Gamma(\nu)} \cdot
\max_{s\in[-1,1], \,t\in[-1,1]^{d-1}}
\sum_{p=1}^\infty \sum_{q\in \Z^{d-1}}
f_\nu\biggl(\frac{\sqrt{2\nu}}{\ell}
\sqrt{(p/h - s)^2 + |q/h - t|^2}
  \biggr)
\nonumber
\\
&&\quad \le\;
2d \frac{2^{1-\nu}}{\Gamma(\nu)} f_\nu(4\nu) e^{2\nu}
\cdot
\!\!\!
\max_{s\in[-1,1], \, t\in[-1,1]^{d-1}}
\sum_{p=1}^\infty \sum_{q\in \Z^{d-1}}
\exp \biggl(
-\frac{\sqrt{\nu}}{\sqrt{2}\ell} \sqrt{(p/h - s)^2 + |q/h - t|^2}
\biggr)
\nonumber
\eea
where in the last step we applied the exponential decay bound
\eqref{eq:besselprop} to each term in the sum.
This is valid since no distance from the kernel origin (square root in the above)
is less than
$(h^{-1}-1)/\ell$, which is at least $\sqrt{8\nu}$ by the hypothesis on $h$.
We now lower-bound the square-root via
$\|y\|_2 \ge \|y\|_1 / \sqrt{d}$ for any $y\in\R^d$, which follows
from Cauchy--Schwarz.
The product now separates along dimensions, so the above is bounded by
\be
2d \frac{2^{1-\nu}}{\Gamma(\nu)} f_\nu(4\nu) e^{2\nu}
\cdot
\biggl(
\max_{s\in[-1,1]}
\sum_{p=1}^\infty
e^{-\sqrt{\frac{\nu}{2d}}\frac{1}{\ell} (p/h-s)}
\biggr)
\biggl(
\max_{t\in[-1,1]}
\sum_{q\in \Z}
e^{-\sqrt{\frac{\nu}{2d}}\frac{1}{\ell} |q/h-t|}
\biggr)^{d-1}
~.
\label{I1I2}
\ee
In the first sum each term is maximized at $s=1$,
so writing $p'=p-1$ we bound that sum geometrically by
$$
e^{-\sqrt{\frac{\nu}{2d}}\frac{h^{-1}-1}{\ell}}
\sum_{p'=0}^\infty e^{-\sqrt{\frac{\nu}{2d}}\frac{p'}{\ell h}}
\;\le\;
 \frac{e^{-\sqrt{\frac{\nu}{2d}}\frac{h^{-1}-1}{\ell}}}{1-e^{-\sqrt{\frac{\nu}{2d}}\frac{1}{\ell h}}}
\;\le\;
2 e^{-\sqrt{\frac{\nu}{2d}}\frac{h^{-1}-1}{\ell}}
$$
where in the last step we used the hypothesis
$\ell\le (\log 2)^{-1}\sqrt{\nu/2d}$ and $h<1$ to upper-bound the
geometric factor by $1/2$.

The second sum over $q$ in \eqref{I1I2} is bounded by its value for $t=0$,
because by the Poisson summation formula \eqref{psf}
it is equal for any $t\in\R$
to $h \sum_{j\in\Z} e^{2\pi i t h j}\frac{2\beta}{\beta^2+(2\pi h j)^2}$
where $\beta = \sqrt{\nu/2}\ell^{-1}$.
This relies on the Fourier transform
of $e^{-\beta|t|}$ being the everywhere-positive function
$\frac{2\beta}{\beta^2+(2\pi\xi)^2}$.
Thus we set $t=0$ in this second sum, write it as two geometric series
with geometric factor again at most $1/2$, which upper bounds the sum by $3$.
Substituting the above two sum bounds into \eqref{I1I2} proves \eqref{Calias}.
\end{proof}

\begin{proof}[Proof of truncation bound \eqref{Ctrunc}] % ppppppppppppppppppppppppppppp
  Now we use Lemma \ref{l:powersum} of half-space lattice sums
  to complete the proof of Theorem \ref{t:Cerr}.
  Noting that $\hat{k} = \hat{C}_{\nu,\ell}$
  from \eqref{Crecall} is always positive, we may drop the phases to get a uniform
  upper bound,
  \bea
  \Biggl| h^{d} \! \sum_{\substack{j \in \mathbb{Z}^{d} \\ j \not \in J_{m}}} \hat{C}_{\nu,\ell}(j h)e^{2\pi ih  \inner{j}{x}} \Biggr|
  &\le&
  h^{d}\!\!\sum_{j \in \Z^d \backslash J_m} \hat{C}_{\nu}(jh, \ell)
  =
  \hat{c}_{d,\nu} (h\ell)^d \!\!
  \sum_{j \in \Z^d \backslash J_m}(2\nu + |2\pi\ell h j|^2)^{-\nu-d/2}
  \nonumber
  \\
  &\le&
  \frac{\hat{c}_{d,\nu}}{(2\pi)^{2\nu+d}}
  (h\ell)^{-2\nu} \sum_{j \in \Z^d \backslash J_m} |j|^{-2\nu-d}
  \nonumber \\
  &\le&
  \frac{\hat{c}_{d,\nu}}{(2\pi)^{2\nu+d}}
  (h\ell)^{-2\nu} \cdot 2d \sum_{n>m} \sum_{q \in \Z^{d-1}} (n^2 + |q|^2)^{-\nu-d/2}
  \nonumber \\
  &=&
  \frac{\hat{c}_{d,\nu}}{(2\pi)^{2\nu+d}}  \frac{2d \,I(d,\nu,m)}{(h \ell)^{2\nu}}
  \;\le\;
  \frac{\hat{c}_{d,\nu}}{(2\pi)^{2\nu+d}} \frac{5^{d-1}}{2\nu} \frac{2d}{(h \ell m)^{2\nu}}
    \nonumber
  .
  \eea
  Here the third inequality follows from noting (similarly to the previous proofs)
  that the sum over $j \in \Z^{d} \backslash J_{m} $ is
  bounded by $2d$ lattice half-spaces of the form $n>m$, $q \in \Z^{d-1}$.
  The last inequality follows from Lemma~\ref{l:powersum}. Substituting \eqref{hatc}
  gives \eqref{Ctrunc}; the theorem is proved.
\end{proof}

As with the SE kernel, this theorem leads to a simple rule to set $h$ and $m$ to guarantee a user-defined absolute
kernel approximation error. In the following we restrict to small dimension, and
use that
the $\nu$-dependent middle factor in \eqref{Calias} never exceeds $3/8$,
and $(2\Gamma(\nu+d/2)/\nu\Gamma(\nu))^{1/2\nu} / \sqrt{2} < 1.6$ for $\nu\ge1/2$, $d\le 3$.

\begin{cor}[Discretization parameters $(h,m)$ to guarantee uniform Mat\'ern kernel
    accuracy $\eps$]
  \label{c:matpars}
  Let the dimension $d$ be 1, 2, or 3.
  Let $k=C_{\nu,\ell}$ be the Mat\'ern kernel with $\nu\ge 1/2$ and
  $\ell\le (\log 2)^{-1}\sqrt{\nu/2d}$ as above.
  Let $\eps>0$.
  Set $h \le \big(1 + \ell \sqrt{2d/\nu} \log (d\,3^d/\eps) \big)^{-1}$,
  then the aliasing error is no more than $\eps/2$.
  In addition, set
  $m \ge (d\,5^{d-1}/\pi^{d/2} \eps)^{1/2\nu} \cdot (1.6)\sqrt{\nu}/\pi h \ell$,
  then the truncation error is no more than $\eps/2$,
  so that the error obeys $|\tilde k(x)-k(x)| \le \eps$ uniformly over $x \in [-1,1]^d$.
\end{cor}

Note that holding $\nu$, $\ell$ and $h$ fixed, $m = \bigO(1/\eps^{1/2\nu})$ as expected from truncating the Mat\'ern
Fourier transform with algebraic decay $1/|\xi|^{2\nu+d}$ (see \eqref{Crecall}).
Instead holding tolerance $\eps$ fixed, $m=\bigO(1/\ell)$ as $\ell\to 0$,
as expected from the growing
number of oscillations in the interpolant across the linear extent of the domain.
The above corollary should be compared with \cite[(1.11)]{bachmayr20},
where $\gamma$ plays the role of $h^{-1}$.
In order to minimize the $m$ used in practice,
instead of the above rigorous parameters choices
we recommend more forgiving heuristics that we state in the next section.

\begin{rmk}
  Both Theorems \ref{t:Gerr} and \ref{t:Cerr} have the very mild restrictions that $\ell$
  be smaller than some $\bigO(1)$ constant. These are in practice irrelevant
  because the domain $D$ is also of size 1 in each dimension,
  and in all applications known to us $\ell$ is set substantially smaller than the
  domain size (otherwise the prior covariance is so long-range that the regression output would be nearly constant over the domain).
\end{rmk}
\begin{rmk}
  Theorems \ref{t:Gerr} and \ref{t:Cerr} have all prefactors explicit.
  %  important for practical choices of $h$ ad $m$.
  It may be possible to improve the prefactors of the
  form $d c^d$ where $2\le c \le 5$,
  since these are due to overcounting where half-spaces overlap
  and bounds on sums over $\Z^{d-1}$ that could be improved.
  The $1/\sqrt{d}$ factor in the exponential in \eqref{Calias} might also be removable
  by using {\em partial} Poisson summation.
%  In the interests of clarity, and since $d\le3$ in this work, we believe that more elaborate bounds are not worth the effort.
\end{rmk}

\bfi   % fffffffffffffffffffffffffffffffffffffffffffffffffffffffffffffffffffffffffffff
    {\hspace{-.6in}\ig{width=1.2\textwidth}{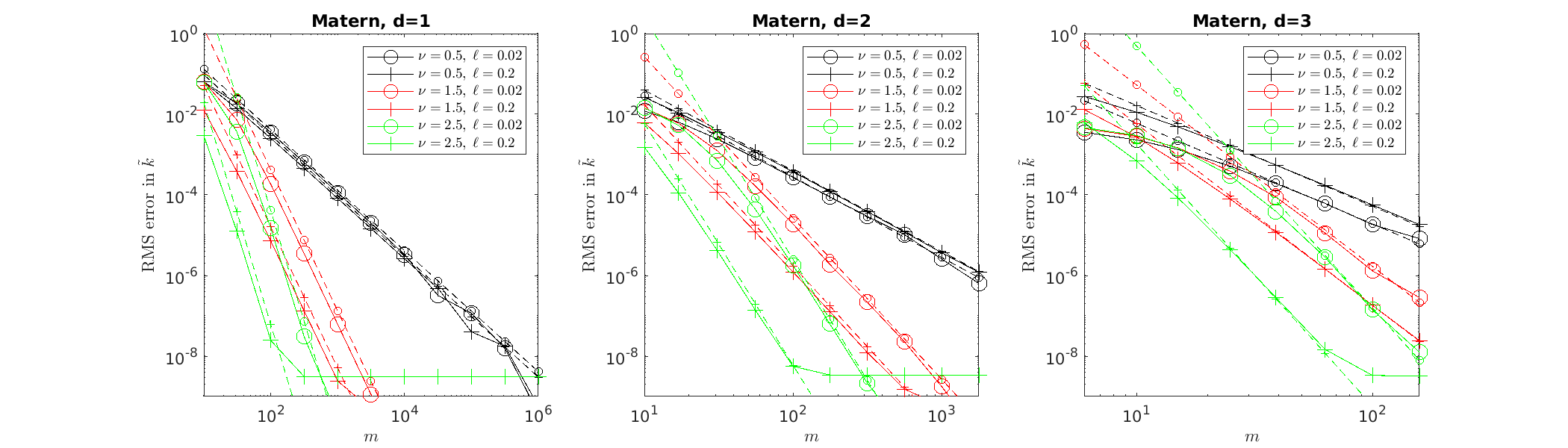}}
    % *** re-crop? png exported with too much white space.
\ca{Estimated root mean square approximation error for the Mat\'ern kernel
  in dimensions $d=1,2,3$ with various parameters, compared to the heuristic
  $\tilde{\eps}$ of \eqref{Kheur} with prefactor as in Remark~\ref{r:heur}.
  The proposed equispaced Fourier basis is used.
  RMS error is found via high-order accurate quadrature
  for the double integral \eqref{doubleint}
  rewritten as $\int_{[-1,1]^d} v(z) |\tilde k(z)-k(z)|^2 dz$ where $v$ is the
  autocorrelation of $\rho$, for the choice $\rho\equiv 1$ in $D$.
  For each choice of $\nu$ and $\ell$, the
  solid line shows the error, while the dotted line shows the heuristic.
  $h$ was chosen via \eqref{hheur} to achieve aliasing error $\eps=10^{-8}$.
}{f:materr}
\efi

% FFFFFFFFFFFFFFFFFFFFFFFFFFFFFFFFFFFFFFFFFFFFFFFFFFFFFFFFFFFFFFFFFFFFFFFFFFFFFFFFFFF
\section{Mat\'ern covariance matrix approximation error} % norms}
\label{s:frob}

For sufficiently non-smooth Mat\'ern kernels, such as $\nu\le 3/2$,
the uniform truncation error bound \eqref{Ctrunc} dominates and gives slow algebraic
convergence $\bigO(1/m^{2\nu})$, due to slow decay in Fourier space.
Yet we have observed that in practice this bound is overly pessimistic
when it comes to the more relevant
{\em root mean square} error of covariance matrix elements,
leading to wasted computational effort.
% (they would require a conspiracy in the choice of $\{x_n\}_{n=1}^N$.
%In \cite{efgp_comput}, we  proposed the following heuristic for the Mat\'ern kernel only for $\| \tilde K - K \|_{F}$. 
We instead propose (and use \cite[Sec.~4.2]{efgp_comput})
the following heuristic with faster convergence $\bigO(1/m^{2\nu+d/2})$.
\begin{conj}[equispaced Fourier Mat\'ern covariance matrix error]
  \label{c:Kheur}
  Let the points $x_1,\dots,x_N$ be iid drawn from some bounded
  probability density function $\rho$ with support in $D=[0,1]^d$.
  Let the Mat\'ern kernel with parameters $\nu$ and $\ell$ be approximated
  by equispaced Fourier modes as in Theorem~\ref{t:Cerr}, with $h$ and $m$
  chosen so that the aliasing error is negligible compared to the truncation error.
  Then with high probability as $N\to\infty$,
  \be
  \| \tilde K - K \|_F \le N \tilde{\eps},
  \qquad \mbox{the root mean square error being} \quad  \tilde{\eps} = \frac{\tilde{c}_{d,\nu,\rho}}{\ell^{2\nu} (hm)^{2\nu+d/2}},
  \label{Kheur}
  \ee
  for some constant $\tilde{c}_{d,\nu,\rho}$ independent of $N$, $\ell$, $h$, and $m$.
\end{conj}

\begin{justification}
%Here we give a heuristic derivation of Conjecture~\ref{c:Kheur}.
Regardless of the kernel or its approximation method, the expectation (over
data point realizations) of the squared Frobenius norm is
\be
\mathbb{E} \|\tilde K- K\|^2_F =
\mathbb{E} \sum_{n,n'=1}^N |\tilde K_{n,n'}- K_{n,n'}|^2
= N^2 \int_D \int_D  |\tilde k(x-x')-k(x-x')|^2 \rho(x)\rho(x') dx dx'.
\label{doubleint}
\ee
Now substituting the dominant truncation part of the pointwise error formula
\eqref{alias_trunc}, and changing variable to $z=x-x'$ which ranges over the set $D-D = [-1,1]^d$,
with $dxdx' = dz dx'$,
we get
\bea
\mathbb{E} \|\tilde K- K\|^2_F &\approx &
N^2 h^{2d}\sum_{j,j' \notin J_m} \hat k(hj) \hat k(hj')
\int_{[-1,1]^d} \!\!\! e^{2\pi i \inner{h(j-j')}{z}}
  \left(\int_D \rho(z+x')\rho(x')dx'\right)
  dz
  \nonumber \\
  &=&
  N^2 h^{2d}\sum_{j,j' \notin J_m} \hat k(hj) \hat k(hj') \, |\hat\rho(h(j-j'))|^2
  ,
\label{WK}
\eea
where the last step used the Wiener--Khintchine theorem for the Fourier transform of
the autocorrelation of $\rho$.
In a mean-square sense with respect to angle we expect Fourier decay
$\hat\rho(\xi) = \bigO(1/|\xi|^{(1+d)/2})$,
even if $\rho$ has discontinuities
(e.g., see \cite{brandolini03} for the case of $\rho\equiv 1$ in $D$,
and we may approximate $\rho$ by a linear combination of
such characteristic functions of convex sets).
% *** need more general rho case refs here? can't find.
Since $|\hat\rho(hj)|^2$ is then summable over $j\in\Z^d$,
and the small $j,j'$ terms dominate (as in the Gibbs phenomenon),
we expect that there is a constant $c_\rho>0$ independent of $m$ such that
$$
h^{2d}\sum_{j,j' \notin J_m} \hat k(hj) \hat k(hj') \, |\hat\rho(h(j-j'))|^2
\; \le \; c_\rho h^{2d}\sum_{j \notin J_m} |\hat k(hj)|^2
\; = \; \bigO\bigl(1/\ell^{4\nu}(hm)^{4\nu + d}\bigr)
,
$$
 where in the last step we used the decay
 of $\hat k(hj)$ from \eqref{Crecall},
 with the sum losing one power of $d$ as in the proof of Theorem \ref{t:Cerr}. Finally, by the central limit theorem we expect, with high probability
as $N\to\infty$, that $\|\tilde K - K\|_F^2$ tends to its expectation,
justifying \eqref{Kheur}.
\end{justification}
A rigorous proof of the conjecture, even for the easiest case $\rho\in C_0^\infty(D)$,
is an open problem.
%that we leave for future work.
We note that related work exists in the variational GP setting \cite{burt19rates}.
Although the iid assumption on data points cannot be justified
in many settings (e.g., satellite data), we find the
conjecture very useful to set numerical parameters even in such cases.
We summarize the resulting
empirically good parameter choices in the following remark.

\begin{rmk}[Discretization parameters $(h,m)$ to achieve empirical root-mean-square
    Mat\'ern kernel accuracy $\eps$]
  \label{r:heur}
  By numerical study of the constant-density case $\rho\equiv 1$ in the domain $D=[0,1]^d$,
  we fit the prefactor $\tilde{c}_{d,\nu,1} \approx 0.15/\pi^{\nu+d/2}$ in \eqref{Kheur}.
  Figure~\ref{f:materr} shows that this truncation
  prediction matches to within a fraction of a
  decimal digit the estimated root mean square error $\tilde\eps$.
  Inverting this gives our proposed numerical grid size choice
  \be
  m \;\approx\;
  \frac{1}{h}\biggl( \pi^{\nu+d/2} \ell^{2\nu} \frac{\eps}{0.15}\biggr)^{-1/(2\nu + d/2)}
  \label{mheur}
  \ee
  to achieve root-mean square truncation error around the given $\eps$.
  The scaling $m = \bigO(1/\eps^{1/(2\nu+d/2)})$ is more
  forgiving than the rigorous $\bigO(1/\eps^{1/2\nu})$ of Corollary~\ref{c:matpars},
  resulting in a smaller grid.
  For instance, for $\nu=1/2$ this lowers $M = \bigO(m^d)$, the total number of modes
  to achieve a Frobenius norm of $N\eps$, from
  $M=\bigO(1/\eps^{d})$ to $M=\bigO(1/\eps^{2d/(2+d)})$,
  a significant reduction in numerical effort when $d$ is ``large'' (eg $3$).
  % *** cut, too much?
  
We also find that a practical choice of $h$ to bound aliasing error by a given
$\eps$ is
\be
h \;\approx\; \bigl(1+0.85 (\ell/\sqrt{\nu}) \log 1/\eps \bigr)^{-1},
\qquad 1/2\le \nu \le 5/2.
\label{hheur}
\ee
This is also verified (for a single $\eps$ choice)
by the saturation of error at a minimum around $10^{-8}$ in 
Figure~\ref{f:materr}.
This $h$ is larger than that in Corollary~\ref{c:matpars}, allowing
$M$, hence the computation time, to be further reduced.
\end{rmk}

We provide extensive numerical experiments using these parameter choices 
in \cite{efgp_comput}, but do not dwell on them further, since
the meat of the present work is the rigorous analysis.

% fffffffffffffffffffffffffffffffffffffffffff COND NUM FIGS ffffffffffffffffffffffffff
\begin{figure*}[t]
\centering
\begin{tikzpicture}[scale=0.8]
\centering
\begin{axis}[
    ymin= -1, ymax=8,
    ytick={-2, -1, 1, 3, 5, 7},
    xlabel=$\log_{10}$ N, 
    ylabel=$\log_{10} \kappa$, %\kappa_\tbox{WS}$,
    legend pos = south east,
%    x dir=reverse
  ]
  
\addplot[color = red, mark=square*, line width=0.1mm]
    coordinates
    {
(1, 1.44)
(2, 2.43)
(3, 3.43)
(4, 4.43)
(5, 5.43)
(6, 6.43)
    };
    
\addplot[color = blue, mark=square*, line width=0.1mm]
    coordinates 
    {
(1, 2.05)
(2, 3.05)
(3, 4.05)
(4, 5.05)
(5, 6.05)
(6, 7.05)
    };

\addplot[color = black, mark=square*, line width=0.1mm]
    coordinates 
    {
(1, 1.25)
(2, 2.43)
(3, 3.43)
(4, 4.43)
    };

\addplot[color = green, mark=square*, line width=0.1mm]
    coordinates 
    {
(1, 1.25)
(2, 2.43)
(3, 3.43)
(4, 4.43)
    };

\legend{ $\kappa_\tbox{WS}$, Upper bound (\ref{kappa_bd}), $\kappa_\tbox{FS}$, $\kappa(K + \sigma^2I)$}
\end{axis}
\end{tikzpicture}
\caption{Condition numbers of the weight-space system matrix
  $A_\tbox{WS} := \Xt\X+\sigma^2 I$, the
  approximate function-space system matrix $A_\tbox{FS}:=\X\Xt + \sigma^2 I$, and
  the exact function-space matrix $K + \sigma^2 I$, as a 
  function of the number of data points $N$, for $d=1$.
  The $\kappa_\tbox{FS}$ curve (black) lies completely under the exact $\kappa$ (green)
  curve.
  %The latter was computed using a dense SVD.
  The data points are
  uniform random on $[0, 1]$, for a squared-exponential kernel with
$\ell = 0.1$, and noise $\sigma = 0.3$.}
\label{fig:cond1}
\end{figure*}
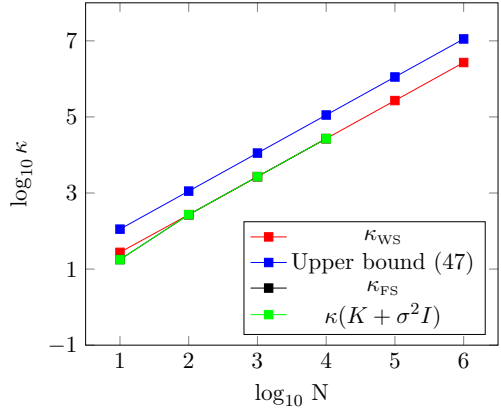

\begin{figure*}[ht]
\begin{subfigure}{0.32 \linewidth}
  \centering
  \begin{tikzpicture}
  \node (img)  {\includegraphics[scale=0.4]{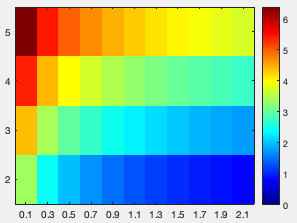}};
  \node[below=of img, node distance=0cm, yshift=1cm,font=\color{black}] {$\sigma$};
  \node[left=of img, node distance=0cm, rotate=90, anchor=center, yshift=-0.7cm,font=\color{black}] {$\log_{10}(N)$};
 \end{tikzpicture}
\caption{$\log_{10} \kappa_\tbox{WS}$}
%\label{fig:kappas}
\end{subfigure}
%\hspace{0.05\linewidth}
\begin{subfigure}{0.32 \linewidth}
  \centering
  \begin{tikzpicture}
  \node (img)  {\includegraphics[scale=0.4]{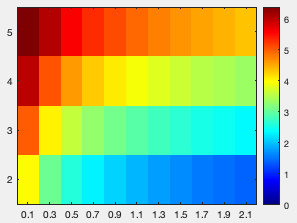}};
  \node[below=of img, node distance=0cm, yshift=1cm,font=\color{black}] {$\sigma$};
%  \node[left=of img, node distance=0cm, rotate=90, anchor=center, yshift=-0.7cm,font=\color{black}] {$\log_{10}(N)$};
 \end{tikzpicture}
\caption{$\log_{10}$ of bound (\ref{kappa_bd})}
%\label{fig:bounds}
\end{subfigure}
%\par\bigskip
\begin{subfigure}{0.32 \linewidth}
  \centering
  \begin{tikzpicture}
  \node (img)  {\includegraphics[scale=0.4]{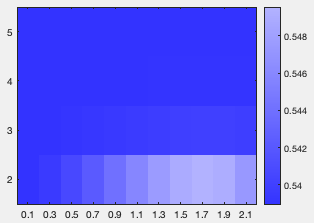}};
  \node[below=of img, node distance=0cm, yshift=1cm,font=\color{black}] {$\sigma$};
%  \node[left=of img, node distance=0cm, rotate=90, anchor=center, yshift=-0.7cm,font=\color{black}] {$\log_{10}(N)$};
 \end{tikzpicture}
\caption{Ratio of $\kappa_\tbox{WS}$ to bound (\ref{kappa_bd})}
%\label{fig:ratios}
\end{subfigure}
\caption{Panel (a) shows the condition number of the weight space system matrix
  $\Xt\X+\sigma^2 I$ for various $N$ and $\sigma^2$.
  (b) shows its upper bound (in the limit $\eps\to0$).
  (c) plots the ratio between the two, which is nearly constant at around 0.54.
  Other parameters are as in Figure~\ref{fig:cond1}. }
\label{fig:cond_heatmaps}
\end{figure*}

% CCCCCCCCCCCCCCCCCCCCCCCCCCCCCCCCCCCCCCCCCCCCCCCCCCCCCCCCCCCCCCCCCCCCCCCCCCCCCCCCCCCCC
\section{Conditioning of function-space and weight-space systems}
\label{s:cond}

In~\cite{efgp_comput} it was observed that the iteration count for conjugate gradient solution with EFGP often grew with the number of data points,
and alarmingly so at smaller tolerance $\eps$.
To grapple with this,
in this final section we present some preliminary analysis that applies
to {\em any} approximate-factorization GP regression method,
connect the weight-space and function-space linear system condition numbers,
and perform a numerical study in the EFGP case.
We do not address preconditioning, but note that it has been beneficial in the GP
context \cite{steinprecond,gardner1,wang19exact}.

Recall that ``exact'' GP regression requires a solution to the function space
linear system
\begin{equation}
(K + \sigma^2 I) \bal = \mbf{y} ,
\end{equation}
and that GP regression using an approximate factorization of the kernel,
in function or weight space, requires solutions to
linear systems with system matrices
\begin{equation} 
A_{\tbox{FS}} = \X\Xt + \sigma^2 I , \qquad A_{\tbox{WS}} = \Xt\X + \sigma^2 I,
\end{equation}
respectively,
where $\X$ is some $N$-by-$M$ design matrix with $\X\Xt=\tilde K \approx K$.
In the special case of EFGP we described the matrix $\Phi$ in
Section~\ref{s:efgp}.
From now we assume that the data size $N$ is large
enough so that $N>M$.

We start with a simple bound for the exact GP regression function space system.
\begin{pro}[Exact function space condition number bound]\label{p:cond}
  Let $k: \R^d \to\R$ be a translationally invariant positive semidefinite
  covariance kernel with $k(\mbf{0})=1$.
  Let $x_1,\dots,x_N\in\R^d$, %be arbitrary data points,
  and $K$ be the $N\times N$ covariance matrix with $ij$th element $k(x_i-x_j)$.
  Then the condition number of the GP function space system matrix obeys
  \be\label{kappa_bd}
  \kappa(K+\sigma^2 I) \; \le \; \frac{N}{\sigma^2} + 1.
  \ee
\end{pro}  
\begin{proof}
  Since $k$ is a positive semidefinite kernel, meaning $\hat{k}$ is nonnegative
\cite[\S4.1]{rasmus1},
  then
  $|k(x)| = |\int_{\R^d} e^{2\pi i\inner{x}{\xi}} \hat{k}(\xi) d\xi | \le \int_{\R^d} \hat{k}(\xi) d\xi = k(\mbf{0}) = 1$.
%  Since $k$ is bounded in absolute value by its value at the origin,
  Thus all entries of $K$ are bounded in magnitude by $1$,
  so $\|K\|_F \le N$.
  Since the spectral norm is bounded by the Frobenius norm,
  the largest eigenvalue of $K$ is no more than $N$
  (or see \cite[p.~207]{wendland05book}),
  and so the
  %largest eigenvalue and
  hence spectral norm of $K + \sigma^2 I$ is no more than $N + \sigma^2$.
  Since $K$ is positive definite, its eigenvalues are nonnegative,
  so that no eigenvalue of $K + \sigma^2I$ is less than $\sigma^2$.
  The proof is completed since $\kappa(K+\sigma^2 I)$ is the ratio of maximum to minimum eigenvalues, because $K$ is symmetric.
\end{proof}

The upper bound is sharp,
since $K$ may come arbitrarily close to the matrix with all entries $1$
when all data points approach the same point.%
\footnote{Alternatively, for fixed $k$ and data domain, as $N\to\infty$ the minimum eigenvalue of $K$ vanishes \cite[p.~54]{steinprecond}.}
The trivial lower bound $\kappa(K+\sigma^2 I)\ge 1$ is also sharp since
$K$ approaches $I$ when all data points move far from each other compared to the kernel width $\ell$.
Only with assumptions on the distribution of data points (their typical separation
compared to $\ell$) could stronger statements be made.
For instance, for fixed $k$ and data domain, as $N\to\infty$ then
$\kappa(K+\sigma^2I)$ grows no slower than $cN/\sigma^2$ for some $c>0$
(this follows from \cite[p.~54]{steinprecond}).
Note that Proposition~\ref{p:cond} could be generalized to the case $k(\mbf{0})\neq 1$
simply by replacing $\sigma^2$ by $\sigma^2/k(\mbf{0})$ in the right-hand side of
\eqref{kappa_bd}.

\begin{rmk}\label{r:illcond}
  The bound \eqref{kappa_bd} is large in practice: for instance in a typical big
problem
with $N=10^7$ and $\sigma = 0.1$, the bound allows $\kappa$ to be $10^9$,
meaning that using single precision there may be \emph{no correct digits} in the
solution $\bal$ to the function space GP linear system \eqref{fs_sys},
and possibly catastrophic cancellation in evaluation of the posterior mean.
\end{rmk}

Two natural questions now arise:
i) Is the GP regression \emph{problem} itself as ill-conditioned as the above
suggests?
ii) Is the weight-space system matrix $A_{\tbox{WS}}$
%that we solve in this work ...ahb: we don't solve it in this work.
similarly conditioned to the (exact) function-space matrix $K+\sigma^2I$ ?
We now show that the answers are respectively ``no'' and ``typically yes.''

\begin{pro}[The GP regression problem %for posterior mean
    at the data points is well-conditioned]
  \label{p:mucond}
  Given data points $x_1,\dots,x_N\in\R^d$, a positive definite kernel,
  and $\sigma>0$,
the absolute condition number of the map from the data vector $\mbf{y}$ to the
posterior mean $\bmu:= \{\mu(x_n)\}_{n=1}^N$ is less than 1.
  \end{pro}
\begin{proof}
  Since the resulting $K$ is positive semidefinite
  \cite[Ch.~4]{rasmus1},
  it may be orthogonally diagonalized
  as $K = \sum_{n=1}^N \lambda_n \mbf{v}_n \mbf{v}_n\tran$ with $\lambda_n\ge 0$
  and $\|\mbf{v}_n\| = 1$,
  so the solution is
  $\bmu = K\bal = K (K + \sigma^2 I)^{-1} \mbf{y} =
  \sum_n \lambda_n(\lambda_n + \sigma^2)^{-1} \mbf{v}_n \mbf{v}_n\tran \mbf{y}$.
  Thus the solution operator has spectral norm $\max_n \lambda_n/(\lambda_n + \sigma^2)<1$.
  \end{proof}

Thus this regression problem remains well conditioned,
even as $N$ grows or $\sigma\to0$,
when (as shown above and below) the system matrix can become very ill-conditioned!
One might worry that an algorithm that solves such an
ill-conditioned system is {\em unstable},
for instance unnecessarily amplifying the kernel error $\eps$.
%system matrix error $\tilde K - K$.
In exact arithmetic, \eqref{muerr} bounds such amplification
by the (large) constant $N/\sigma^2$.
In floating point arithmetic the situation may be more dire, due to
catastrophic cancellation in evaluating $\bmu = \X \bbe$, if $\|\bbe\|$ is large.

For $\mu(x)$ at new targets, even less is known (at least to these authors):
it is unknown whether the absolute condition number of the regression problem is
even $\bigO(1)$,
or whether the extremely large $N^2/\sigma^4$ amplification factor in
the naive bound \eqref{muerrnew} could be reduced.

We now turn to question ii):
how close are the condition numbers of the 
weight-space $A_\tbox{WS}$, approximate function-space $A_\tbox{FS}$, and
exact function-space $K+\sigma^2 I$ system matrices?
We now show, as $\eps\to0$ (good kernel approximation),
that neither
of the two approximate linear systems can be worse conditioned than the exact one.
%(Without detailed knowledge of the kernel we are unable do exclude the possibility that they are better conditioned.)
To interpret the following, recall from \eqref{Kerr} that a pointwise
kernel error of $\eps$ leads to the simple bound $\|\tilde K - K\|\le N\eps$.

\begin{lem}[Approximated linear system condition number bounds]
  \label{l:condWS}
  Let $\X\in\C^{N\times M}$ with $M<N$, such
  that $\tilde K = \X \Xt$ approximates
  $K$ to spectral norm error $\|\tilde K - K\|\le N\eps$.
  Then the approximated function-space condition number denoted by $\kappa_{\tbox{FS}} = \kappa(A_{\tbox{FS}})$, and
  the weight space condition number denoted by $\kappa_{\tbox{WS}} = \kappa(A_{\tbox{WS}})$, both have an upper bound
  \be
    \kappa_{\tbox{FS}},
  \kappa_\tbox{WS}
    \;\le \; \biggl( 1 + \frac{\eps N}{\sigma^2}\biggr)
  \kappa(K + \sigma^2 I) + \frac{\eps N}{\sigma^2} .
  \label{condWS}
  \ee
\end{lem}
\begin{proof}
  Our main tool is eigenvalue perturbation: for each eigenvalue of $K$ there is an
  eigenvalue of $\tilde K$ within a distance of $\eps N$,
  which follows from
  the symmetric case of the Bauer--Fike theorem \cite[Thm.~7.7.2]{golubvanloan}
  and $\|\tilde K - K\|\le \eps N$.
  Since $M<N$, $\tilde K$ has a zero eigenvalue, so the minimum
  eigenvalue of $K$ is $\lambda_\text{min} \in [0, \eps N]$.
  Abbreviating $\kappa:=\kappa(K + \sigma^2 I) = (\lambda_\text{max}+\sigma^2)/(\lambda_\text{min}+\sigma^2)$,
  then $\lambda_\text{max} \le (\sigma^2+\eps N)\kappa - \sigma^2$.
  Again by eigenvalue perturbation, the largest eigenvalue of $\tilde K$
  is no more than $\lambda_\text{max} + \eps N \le (\sigma^2+\eps N)\kappa - \sigma^2 + \eps N$,
  and the same is true for $\Xt\X$ since its nonzero eigenvalues match those of $\tilde K$.
  Thus $\|\Xt\X + \sigma^2 I\| \le (\sigma^2+\eps N)\kappa + \eps N$. Finally, since 
$\|(\Xt\X + \sigma^2 I)^{-1}\|\le\sigma^{-2}$, and $\|(\X\Xt + \sigma^2 I)^{-1}\|=\sigma^{-2}$, the results then follow.
\end{proof}

We now report a test of
empirical condition number growth vs $N$ and $\sigma^{-2}$, for random data in
$d=1$.
Figures \ref{fig:cond1} and \ref{fig:cond_heatmaps} compare
$\kappa(K+\sigma^2 I)$, $\kappa_{\tbox{FS}}$, $\kappa_\tbox{WS}$, and also the
theoretical upper bound $N/\sigma^2 + 1$ from \eqref{kappa_bd}.
The upper bound applies to all three as $\eps\to0$.
The plots show that the three condition numbers are extremely close to each other,
and that the bound overestimates them
by only a factor of roughly $2$, over a wide $(N,\sigma)$ parameter space.
Here we used dense symmetric diagonalization
for the ``exact'' calculations of $\kappa(K+\sigma^2 I)$
and $\kappa_\tbox{FS}$ for $N\le 10^4$,
and EFGP with $\eps$ converged down to $10^{-16}$
(i.e., machine precision) for the other cases.
By Corollary~\ref{c:SEparams} this requires only a small $m<30$,
%($M=61$),
for which dense diagonalization of $A_\tbox{WS}$ is trivial.

% I returned this para to close to arxiv version:
Finally, the above has consequences for the convergence rate of conjugate gradient to
solve either function-space or weight-space linear systems.
For instance, if $\bbe$ is the exact solution to \eqref{WS},
and $\bbe_k$ its approximation at the $k$th CG iteration \cite[\S10.2]{golubvanloan},
\begin{align}\label{85}
  \| \bbe_k - \bbe \| \; = \;
  \bigO\bigg( \bigg(\frac{\sqrt{\kappa_\tbox{WS}} - 1}{\sqrt{\kappa_\tbox{WS}} + 1} \bigg)^k \bigg).
\end{align}
Since $\kappa_\tbox{WS} \lesssim N/\sigma^2$ for small $\eps$,
this gives convergence no slower than $e^{-(2\sigma/\sqrt{N})k}$.
Thus one requires at most $\bigO\bigl(\log(1/\eps) \sqrt{N}/\sigma \bigr)$
iterations to reach a residual $\eps$.

\section{Conclusions and generalizations \label{s:conc}}

In this paper, we provided a detailed error analysis for the equispaced Fourier 
Gaussian process (EFGP) kernel representation of \cite{efgp_comput}. 
The main results (Theorems~\ref{t:Gerr} and \ref{t:Cerr})
gave uniform kernel approximation error bounds for the popular squared-exponential and Mat\'ern kernels, with all constants explicit, in general dimension.
This led to Fourier quadrature grid parameters
that guarantee a desired kernel error
(Corollaries~\ref{c:SEparams} and \ref{c:matpars}).
Since this equispaced Fourier grid is maybe the simplest spectral kernel
approximation,
%of utility in low $d$,
we expect these to find applications in other kernel methods.

For the Mat\'ern kernel with small $\nu$, these uniform error
bounds are in practice pessimistic when it comes to {\em root mean square} errors,
because of the slow Fourier decay of the kernel.
Thus we proposed a conjecture on the root mean square kernel error, and supported it
by numerical tests. A proof, even for iid random
data coming from a smooth density function, remains open.
%(implying a rapidly decaying Fourier transform) 

Finally, we proved an upper bound on the condition number of the
approximate function- and 
weight-space linear systems for arbitrary data distributions,
showing how they approach the ``exact'' GP condition number as the kernel
approximation error vanishes.
We then showed experimentally
that such condition numbers for a simple random data point distribution are
about as ill-conditioned as possible, i.e., within a small factor of $N/\sigma^2$.
Yet, Proposition~\ref{p:mucond} reminds one that
the GP regression {\em problem} itself (at least for the mean at the
data points themselves) is well conditioned.
In short,
{\em an ill-conditioned algorithm appears to be necessary to
  solve a well-conditioned problem},
raising the eyebrows of any numerical analyst.
This motivates the future study of stability (coefficient norm growth and the
resulting rounding
loss) in GP settings, especially in an era of reduced (e.g. half-) precision
arithmetic.
It also suggests the continued study of preconditioners for GP regression
problems.

Many other interesting analysis questions remain,
such as Fourier kernel approximation bounds for other common kernels,
the conditioning of the regression problem
to new targets, and a rigorous lower bound on $\kappa_\tbox{WS}$
(analogous to Lemma~\ref{l:condWS}), which would demand
knowledge of $K$'s smallest eigenvalue.

\section*{Acknowledgments}
The authors are grateful for helpful discussions with
Charlie Epstein and Jeremy Hoskins.
The second author is supported in part by the Alfred P. Sloan Foundation,
the Office of Naval Research, and the NSF. 
The Flatiron Institute is a division of the Simons Foundation.

\bibliographystyle{abbrv}
\bibliography{refs}

\end{document}